\documentclass[10pt]{amsart}

\usepackage{amssymb}
\usepackage{geometry}
\usepackage{graphicx}

\newtheorem{teo}{Theorem}[section]
\newtheorem{coro}[teo]{Corollary}
\newtheorem{lm}[teo]{Lemma}
\newtheorem{prop}[teo]{Proposition}
\newtheorem*{main}{Main Theorem}
\newtheorem*{buser}{Euclidean Buser's inequality}

\theoremstyle{definition}
\newtheorem{defi}[teo]{Definition}
\newtheorem{oss}[teo]{Remark}

\newtheorem*{ack}{Acknowledgments}

\numberwithin{equation}{section}

\usepackage[colorlinks=true,urlcolor=blue, citecolor=red,linkcolor=blue,
linktocpage,pdfpagelabels, bookmarksnumbered,bookmarksopen]{hyperref}
\usepackage[hyperpageref]{backref}

\title[Principal frequencies VS. isoperimetric ratios]{On principal frequencies and isoperimetric ratios\\ in convex sets}

\author[Brasco]{Lorenzo Brasco}
\address[L.\ Brasco]{Dipartimento di Matematica e Informatica
\newline\indent
Universit\`a degli Studi di Ferrara
\newline\indent
Via Machiavelli 35, 44121 Ferrara, Italy}
\email{lorenzo.brasco@unife.it}

\subjclass[2010]{35P15, 49J40, 35J70}
\keywords{Buser's inequality, convex sets, $p-$Laplacian, Cheeger constant, shape optimization}

\begin{document}

\begin{abstract}
On a convex set, we prove that the Poincar\'e-Sobolev constant for functions vanishing at the boundary can be bounded from above by the ratio between the perimeter and a suitable power of the $N-$dimensional measure. This generalizes an old result by P\'olya. As a consequence, we obtain the sharp {\it Buser's inequality} (or reverse Cheeger inequality) for the $p-$Laplacian on convex sets. This is valid in every dimension and for every $1<p<+\infty$.
We also highlight the appearing of a subtle phenomenon in shape optimization, as the integrability exponent varies. 
\end{abstract}

\maketitle
\begin{center}
\begin{minipage}{8cm}
\small
\tableofcontents
\end{minipage}
\end{center}

\section{Introduction}

\subsection{Background}
Let $\Omega\subset\mathbb{R}^N$ be an open bounded convex set, we consider its {\it fundamental frequency} or {\it first eigenvalue of the Dirichlet-Laplacian}, i.e.
\[
\lambda(\Omega):=\inf_{u\in C^\infty_0(\Omega)\setminus\{0\}} \frac{\displaystyle\int_\Omega |\nabla u|^2\,dx}{\displaystyle\int_\Omega |u|^2\,dx}.
\]
An old result by P\'olya (see \cite{Po}), later generalized by Jo\'o and Stach\'o (see \cite{JS}), asserts that $\lambda(\Omega)$ can be bounded from above in a sharp way, by the ratio between the perimeter and the volume of $\Omega$. Namely, it holds
\begin{equation}
\label{polya}
\lambda(\Omega)<\frac{\pi^2}{4}\,\left(\frac{P(\Omega)}{|\Omega|}\right)^2.
\end{equation}
Equality is never attained on bounded convex sets, but the estimate is sharp. A sequence of sets saturating the inequality is indeed given by
\[
\Omega_L=\left(-\frac{L}{2},\frac{L}{2}\right)^{N-1}\times(0,1),\qquad \mbox{ for } L\to +\infty.
\] 
The method of proof by P\'olya is based on the so-called {\it method of interior parallels}. 
\par
Parini in \cite{Pa} recently observed that \eqref{polya} in turn implies the following inequality
\begin{equation}
\label{reverse_parinz}
\lambda(\Omega)<\frac{\pi^2}{4}\,\Big(h_1(\Omega)\Big)^2,
\end{equation}
where $h_1(\Omega)$ is the {\it Cheeger constant} of $\Omega$, defined by
\[
h_1(\Omega):=\inf_{E\subset \Omega} \left\{\frac{P(E)}{|E|}\, :\, |E|>0\right\},
\]
and $P(\,\cdot\,)$ stands for the distributional perimeter of a set in the sense of De Giorgi.
\par
It is useful to recall that, in general, for bounded open sets it holds
\[
\frac{1}{4}\,\Big(h_1(\Omega)\Big)^2<\lambda(\Omega).
\]
This is called {\it Cheeger's inequality}, first proved in \cite{cheeger} in the context of Riemannian manifolds without boundary and in \cite{LW} in the Euclidean setting. Thus we can refer to inequality \eqref{reverse_parinz} as {\it reverse Cheeger's inequality}. 
An estimate of this flavour was first proved by Buser in \cite[Theorem 1.2]{bu} for the Laplace-Beltrami operator on compact Riemannian manifolds without boundary, having positive Ricci curvature. We also refer to \cite{Le} for an alternative proof by Ledoux, which uses the heat semigroup (see also \cite[Theorem 5.2]{Le2} for a finer estimate with a constant independent of the dimension). For this reason, we can also call \eqref{reverse_parinz} {\it Buser's inequality}.
\begin{oss}
In \cite{Pa} inequality \eqref{reverse_parinz} is stated for $N=2$, but it is easy to see that it holds in any dimension. Indeed, the proof of \eqref{reverse_parinz} is just based on:
\begin{itemize}
\item inequality \eqref{polya};
\vskip.2cm
\item the fact that $\lambda(\Omega)\le \lambda (E)$ for every open set $E\subset\Omega$;
\vskip.2cm
\item the convexity of sets $E_\Omega$ attaining the infimum in the definition of $h_1(\Omega)$ (these are called {\it Cheeger sets}).
\end{itemize}
These three facts hold in every dimension.
\end{oss}

\subsection{Main result}
In a nutshell, we can describe the main results of this paper as follows: we generalize P\'olya's inequality \eqref{polya} to the case of the $p-$Laplacian and this in turn permits to generalize Buser's inequality to the $p-$Laplacian. 
\par
In order to describe more precisely our results, for $1<p<+\infty$ and $1\le q\le +\infty$ we introduce the quantity
\[
\lambda_{p,q}(\Omega)=\inf_{u\in C^\infty_0(\Omega)\setminus\{0\}} \frac{\|\nabla u\|_{L^p(\Omega)}^p}{\displaystyle \|u\|_{L^q(\Omega)}^p},
\]
and the one-dimensional Poincar\'e constant
\[
\pi_{p,q}=\min_{u\in W^{1,p}((0,1))\setminus\{0\}} \left\{\frac{\|u'\|_{L^p((0,1))}}{\displaystyle \|u\|_{L^q((0,1))}}\, :\,u(0)=u(1)=0\right\}.
\]
By using the method of interior parallels as in \cite{Po}, we will prove the following result. The generalization of P\'olya's result to the case $p=q$ has been already obtained in \cite{DPG}.
\begin{main}
Let $1<p<+\infty$ and 
\[
\left\{\begin{array}{rl}
1\le q<\dfrac{N\,p}{N-p},& \mbox{ if } p\le N,\\
&\\
1\le q\le +\infty, & \mbox{ if }p>N.
\end{array}
\right.
\] 
Let $\Omega\subset \mathbb{R}^N$ be an open bounded convex set. Then we have\footnote{In the case $q=+\infty$, we use the convention $1/q=0$.}
\begin{equation}
\label{polyapq}
\lambda_{p,q}(\Omega)<\left(\frac{\pi_{p,q}}{2}\right)^p\,\left(\frac{P(\Omega)}{|\Omega|^{1-\frac{1}{p}+\frac{1}{q}}}\right)^p,
\end{equation}
and the inequality is strict. Moreover, the constant is sharp for $q\le p$.
\end{main}
The original statement \eqref{polya} corresponds to take $p=q=2$. Indeed, in this case we have $\pi_{2,2}=\pi$. Thus we observe that the result is already new for the Laplacian, i.e. for $p=2$ and $q\not =2$.
\vskip.2cm\noindent
As in the case $p=2$, we can use the Main Theorem with $q=p$ and get the following generalization of \eqref{reverse_parinz}.
\begin{buser}
Let $1<p<+\infty$ and let $\Omega\subset \mathbb{R}^N$ be an open convex set. Then we have
\[
\lambda_p(\Omega)\le \left(\frac{\pi_p}{2}\right)^p\,\Big(h_1(\Omega)\Big)^p,
\]
and the inequality is sharp.
\end{buser}
Let us now go back to inequality \eqref{polyapq}. We observe that the quantity
\[
\Omega\mapsto \left(\frac{|\Omega|^{1-\frac{1}{p}+\frac{1}{q}}}{P(\Omega)}\right)^p\,\lambda_{p,q}(\Omega),
\]
is scale invariant. Thus we can rephrase the previous result by saying that, if we set
\begin{equation}
\label{SOP}
\lambda^*_{p,q}=\sup\left\{\left(\frac{|\Omega|^{1-\frac{1}{p}+\frac{1}{q}}}{P(\Omega)}\right)^p\,\lambda_{p,q}(\Omega)\, :\, \Omega\subset\mathbb{R}^N \mbox{ open bounded convex set}\right\},
\end{equation}
then 
\[
\lambda^*_{p,q}<\left(\frac{\pi_{p,q}}{2}\right)^p\,
\]
and for $q\le p$ {\it there are no optimal shapes}, only maximizing sequences. These are given for example by
\[
\Omega_L=\left(-\frac{L}{2},\frac{L}{2}\right)^{N-1}\times(0,1),\qquad \mbox{ for } L\to +\infty,
\]
see the proof of the Main Theorem.
\subsection{The case $q>p$}
The fact that we can prove sharpness only for $q\le p$ is not due to a defect in the method of proof, but to the presence of a weird phenomenon. Indeed, we will show in Theorem \ref{teo:existence} that for $q>p$ {\it the situation abruptly changes}: the shape optimization problem \eqref{SOP} {\it does admit a solution}. Thus, the upper bound given by the one-dimensional problem is no more optimal for $q>p$. The problem of providing the sharp value $\lambda^*_{p,q}$ seems to be a challenging task. We make some comments and give some partial results in Remark \ref{oss:palle} and Proposition \ref{prop:palle} below.
\begin{oss}
A similar phenomenon has been observed by Nitsch, Trombetti and the author, in the {\it Neumann case}, i.e. for the Poincar\'e constant
\[
\mu_{p,q}(\Omega)=\inf_{u\in C^1(\overline{\Omega})\setminus\{0\}} \left\{\frac{\displaystyle\int_\Omega |\nabla u|^p\,dx}{\displaystyle\left(\int_\Omega |u|^q\,dx\right)^\frac{p}{q}}\,:\, \int_\Omega |u|^{q-2}\,u\,dx=0\right\},
\]
and the related shape optimization problem
\[
\sup\left\{\Big(\mathrm{diam}(\Omega)\Big)^{p-N+N\,\frac{p}{q}}\,\mu_{p,q}(\Omega)\, :\, \Omega\subset\mathbb{R}^N \mbox{ open bounded convex set}\right\}.
\]
We refer to \cite[Theorem 4.4]{BNT} for more details.
\end{oss}

\subsection{Plan of the paper}

We set the notation in Section \ref{sec:preliminaries}, then in Section \ref{sec:proof} we give the proof of the Main Theorem. We discuss some of its consequences in Section \ref{sec:consequence}, notably we prove Buser's inequality. With Section \ref{sec:pq} we try to shed some light on the shape optimization problem \eqref{SOP}: the main result in this part is Theorem \ref{teo:existence}.
\par
The paper is complemented with two appendices: the first one concerning the one-dimensional constant $\pi_{p,q}$; the second one proving some estimates (containing inradius, perimeter, volume and diameter) for convex sets.

\begin{ack} 
The author thanks Eleonora Cinti for some discussions on Proposition \ref{lm:geometrico} and Michel Ledoux for having pointed out reference \cite{Le2}. Bozhidar Velichkov is gratefully acknowledged for having revitalized author's interest in Buser's inequality, after a period of neglect. This paper would have not been written without his impulse, we wish to warmly thank him. 
\par
This work has been finalized during the conferences ``{\it Variational and PDE problems in Geometric Analysis}'' and ``{\it Recent advances in Geometric Analysis}'' held in June 2018 in Bologna and Pisa, respectively. The organizers and the hosting institutions are kindly acknowledged.
\end{ack}

\section{Preliminaries}

\label{sec:preliminaries}

\subsection{Notation}

For an open set $\Omega\subset\mathbb{R}^N$, we indicate by $|\Omega|$ its $N-$dimensional Lebesgue measure. We use the standard notations
\[
B_R(x_0)=\{x\in\mathbb{R}^N\, :\, |x-x_0|<R\}\qquad \mbox{ and }\qquad \omega_N=|B_1(0)|.
\]
For an open set $\Omega\subset\mathbb{R}^N$ with Lipschitz boundary, we define the {\it distance function}
\[
\delta_\Omega(x)=\inf_{y\in\partial\Omega} |x-y|,\qquad x\in\Omega,
\]
and
\[
R_\Omega=\sup_{x\in\Omega} \delta_\Omega(x).
\]
The last quantity is called {\it inradius of $\Omega$} and it coincides with the radius of the largest ball inscribed in $\Omega$. 

\subsection{Poincar\'e-Sobolev constants}
For every $1<p<+\infty$, we set
\[
p^*=\left\{\begin{array}{cc}
\dfrac{N\,p}{N-p},& \mbox{ if }p<N,\\
&\\
+\infty, & \mbox{ if }p\ge N.
\end{array}
\right.
\]
Then if $\Omega\subset\mathbb{R}^N$ is an open set, we define its sharp Poincar\'e-Sobolev constant
\[
\lambda_{p,q}(\Omega)=\inf_{u\in C^\infty_0(\Omega)} \left\{\int_\Omega |\nabla u|^p\,dx\, :\, \|u\|_{L^q(\Omega)}=1\right\},\qquad \mbox{맍or }\left\{\begin{array}{rl}
1\le q<p^*,& \mbox{ if } p\le N,\\
&\\
1\le q\le +\infty, & \mbox{ if }p>N.
\end{array}
\right.
\]
For $p=q$, we will use the convention of writing 
\[
\lambda_p(\Omega) \qquad \mbox{ in place of }\qquad \lambda_{p,p}(\Omega).
\] 
Whenever $\Omega\subset\mathbb{R}^N$ is bounded, the infimum above is attained on the homogeneous Sobolev space $\mathcal{D}^{1,p}_0(\Omega)$. The latter is defined as the completion of $C^\infty_0(\Omega)$ with respect to the norm
\[
u\mapsto \left(\int_\Omega |\nabla u|^p\,dx\right)^\frac{1}{p}.
\]
For $q=1$, the quantity $1/\lambda_{p,1}(\Omega)$ is also called {\it $p-$torsional rigidity}. This is usually denoted by
\[
T_p(\Omega)=\frac{1}{\lambda_{p,1}(\Omega)}=\sup_{u\in C^\infty_0(\Omega)\setminus \{0\}} \frac{\displaystyle\left(\int_\Omega u\,dx\right)^p}{\displaystyle \int_{\Omega}|\nabla u|^p\,dx}.
\]
For an open set $\Omega\subset\mathbb{R}^N$, we consider its {\it Cheeger constant}, defined by
\[
h_1(\Omega):=\inf_{E\subset \Omega} \left\{\frac{P(E)}{|E|}\, :\, |E|>0 \mbox{ and } E \mbox{ is bounded}\right\}.
\]
A minimizing set for the problem above is called {\it Cheeger set} of $\Omega$.
\vskip.2cm\noindent
For later reference, we recall the following simple result. We point out that we are not assuming anything on the open sets, thus the proof is slightly more complicated than what one could think at first sight.
\begin{lm}[Right continuity in $q$]
\label{lm:continuity}
Let $1<p<+\infty$, for every $\Omega\subset\mathbb{R}^N$ open set, we have
\[
\lim_{q\searrow p} \lambda_{p,q}(\Omega)=\lambda_p(\Omega).
\]
\end{lm}
\begin{proof}
We first observe that for $p<q<p^*$ we have
\[
\lambda_{p,q}(\Omega)>0 \qquad \Longleftrightarrow\qquad \lambda_{p}(\Omega)>0,
\]
see \cite[Theorem 15.4.1]{maz} (and also \cite[Remark 4.4]{BFR} for a different proof).
We can thus assume that $\Omega$ is such that $\lambda_p(\Omega)>0$, otherwise there is nothing to prove. For every $\varepsilon>0$, we take $u_\varepsilon\in C^\infty_0(\Omega)$ such that
\[
\int_\Omega |\nabla u_\varepsilon|^p\,dx<\lambda_p(\Omega)+\varepsilon\qquad \mbox{마nd }\qquad \int_\Omega|u_\varepsilon|^p\,dx=1.
\]
Let us call $\Omega_\varepsilon$ the support of $u_\varepsilon$, then by H\"older inequality we have for $q>p$
\[
\lambda_{p,q}(\Omega)\le \frac{\displaystyle\int_\Omega |\nabla u_\varepsilon|^p\,dx}{\displaystyle\left(\int_\Omega |u_\varepsilon|^q\,dx\right)^\frac{p}{q}}\le \frac{\displaystyle\int_\Omega |\nabla u_\varepsilon|^p\,dx}{\displaystyle \int_\Omega |u_\varepsilon|^p\,dx}\,|\Omega_\varepsilon|^{1-\frac{p}{q}}<|\Omega_\varepsilon|^{1-\frac{p}{q}}\,(\lambda_p(\Omega)+\varepsilon).
\]
This implies that
\[
\limsup_{q\searrow p} \lambda_{p,q}(\Omega)\le \lambda_p(\Omega)+\varepsilon,
\]
for every $\varepsilon>0$. If we now prove that
\[
\liminf_{q\searrow p} \lambda_{p,q}(\Omega)\ge \lambda_p(\Omega),
\]
this would give the desired conclusion. We fix an exponent $p<q_0<p^*$, then for every $\varepsilon>0$ and every $p<q<q_0$, we take $u_{\varepsilon,q}\in C^\infty_0(\Omega)$ such that
\[
\int_\Omega |\nabla u_{\varepsilon,q}|^p\,dx<\lambda_{p,q}(\Omega)+\varepsilon\qquad \mbox{마nd }\qquad \int_\Omega|u_{\varepsilon,q}|^q\,dx=1.
\]
By interpolation in Lebesgue spaces, we have
\begin{equation}
\label{interpol}
\left(\int_\Omega|u_{\varepsilon,q}|^q\,dx\right)^\frac{p}{q}\le\left(\int_\Omega|u_{\varepsilon,q}|^p\,dx\right)^{\vartheta(q)}\,\left(\int_\Omega|u_{\varepsilon,q}|^{q_0}\,dx\right)^{\frac{p}{q_0}\,(1-\vartheta(q))},
\end{equation}
with 
\[
\vartheta(q)=\frac{p}{q}\,\frac{q_0-q}{q_0-p}.
\]
On the other hand, by the choice of $u_{\varepsilon,q}$ and the definition of $\lambda_{p,q_0}(\Omega)$ we have
\[
\lambda_{p,q_0}(\Omega)\,\left(\int_\Omega |u_{\varepsilon,q_0}|^{q_0}\,dx\right)^\frac{p}{q_0}\le\int_\Omega |\nabla u_{\varepsilon,q}|^p\,dx<(\lambda_{p,q}(\Omega)+\varepsilon)\,\left(\int_{\Omega} |u_{\varepsilon,q}|^q\,dx\right)^\frac{p}{q},
\]
that is, we have the reverse H\"older inequality
\[
\left(\int_\Omega |u_{\varepsilon,q}|^{q_0}\,dx\right)^\frac{p}{q_0}<\frac{\lambda_{p,q}(\Omega)+\varepsilon}{\lambda_{p,q_0}(\Omega)}\,\left(\int_\Omega |u_{\varepsilon,q}|^{q}\,dx\right)^\frac{p}{q}.
\]
We can spend this information in \eqref{interpol}, so to obtain
\begin{equation}
\label{bordello}
\left(\int_\Omega|u_{\varepsilon,q}|^q\,dx\right)^\frac{p}{q}\le\left(\int_\Omega|u_{\varepsilon,q}|^p\,dx\right)\,\left(\frac{\lambda_{p,q}(\Omega)+\varepsilon}{\lambda_{p,q_0}(\Omega)}\right)^{\frac{1}{\vartheta(q)}-1}.
\end{equation}
By using $u_{\varepsilon,q}$ as an admissible test function for $\lambda_p(\Omega)$ and \eqref{bordello}, we get
\[
\begin{split}
\lambda_{p}(\Omega)&\le \frac{\displaystyle\int_\Omega |\nabla u_{\varepsilon,q}|^p\,dx}{\left(\displaystyle\int_\Omega|u_{\varepsilon,q}|^q\,dx\right)^\frac{p}{q}}\,\left(\frac{\lambda_{p,q}(\Omega)+\varepsilon}{\lambda_{p,q_0}(\Omega)}\right)^{\frac{1}{\vartheta(q)}-1}\\
&\le (\lambda_{p,q}(\Omega)+\varepsilon)^\frac{1}{\vartheta(q)}\,\left(\frac{1}{\lambda_{p,q_0}(\Omega)}\right)^{\frac{1}{\vartheta(q)}-1}.
\end{split}
\] 
By observing that $\vartheta(q)$ goes to $1$ as $q$ goes to $p$, we get from the last estimate
\[
\lambda_p(\Omega)=\lim_{q\searrow p} \lambda_{p}(\Omega)^{\vartheta(q)} \le \liminf_{q\searrow p} \left[(\lambda_{p,q}(\Omega)+\varepsilon)\,\left(\frac{1}{\lambda_{p,q_0}(\Omega)}\right)^{1-\vartheta(q)}\right]= \liminf_{q\searrow p} \lambda_{p,q}(\Omega)+\varepsilon.
\]
As $\varepsilon>0$ is arbitrary, we get the desired conclusion.
\end{proof}
\begin{oss}[Left continuity in $q$]
If we do not take any assumption on the open set $\Omega$, in general {\it it is not true} that 
\[
\lim_{q\nearrow p} \lambda_{p,q}(\Omega)=\lambda_p(\Omega).
\]
As a simple counter-example, one can take the slab
\[
\Omega=\mathbb{R}^{N-1}\times (0,1).
\]
Indeed, in this case we have (see Lemma \ref{lm:slab} below)
\[
\lambda_p(\mathbb{R}^{N-1}\times(0,1))>0.
\]
On the other hand, for $q<p$ we have that 
\[
\lambda_{p,q}(\Omega)>0 \qquad \Longleftrightarrow\qquad \mbox{ the embedding }\mathcal{D}^{1,p}_0(\Omega)\hookrightarrow L^q(\Omega) \mbox{ is compact},
\]
see \cite[Theorem 1.2]{BR}.
Of course, the last property can not be true for the slab $\mathbb{R}^{N-1}\times (0,1)$, because the set is invariant with respect to translations with respect to the first $N-1$ variables. This implies that
\[
\lim_{q\nearrow p}\lambda_{p,q}(\mathbb{R}^{N-1}\times (0,1))=0< \lambda_{p}(\mathbb{R}^{N-1}\times (0,1)).
\] 
A sufficient condition ensuring left continuity in $q$ is $|\Omega|<+\infty$, which is however not necessary.
\end{oss}
The next property can be proven quite easily by appealing to H\"older inequality, we omit the details.
\begin{lm}[Monotonicity]
\label{lm:monotone}
Let $1<p<+\infty$ and let $\Omega\subset\mathbb{R}^N$ be an open set such that $|\Omega|<+\infty$ and $P(\Omega)<+\infty$. Then 
\[
q\mapsto \lambda_{p,q}(\Omega)\,\left(\frac{|\Omega|^{1-\frac{1}{p}+\frac{1}{q}}}{P(\Omega)}\right)^p,
\]
is a monotone non-increasing function.
\end{lm}
Finally, for $1<p<+\infty$ and $1\le q\le +\infty$ we recall the notation already used in the Introduction
\[
\pi_{p,q}=\min_{u\in W^{1,p}((0,1))\setminus\{0\}} \left\{\frac{\|u'\|_{L^p((0,1))}}{\displaystyle \|u\|_{L^q((0,1))}}\, :\,u(0)=u(1)=0\right\}.
\]
Here as well, we will use the shortcut notation $\pi_p$ in the case $q=p$.
By scaling, we easily get for every $L>0$
\begin{equation}
\label{1dim}
\min_{u\in W^{1,p}((0,L))\setminus\{0\}} \left\{\frac{\|u'\|_{L^p((0,L))}^p}{\displaystyle \|u\|_{L^q((0,L))}^p}\, :\,u(0)=u(L)=0\right\}=\frac{\Big(\pi_{p,q}\Big)^p}{L^{p-1+\frac{p}{q}}}.
\end{equation}
\begin{oss}[Some explicit values]
\label{oss:costanti}
We have already observed that $\pi_{2}=\pi$ 
and the corresponding extremals are given by
\[
u(t)=C\,\sin\left(\pi\,t\right),\qquad \mbox{ where } C\in\mathbb{R}\setminus\{0\}.
\]
We also have
\[
\pi_{p,1}=2\,\left(\frac{2\,p-1}{p-1}\right)^\frac{p-1}{p},
\]
and the corresponding extremals are given by
\[
u(t)=C\,\left(\left(\frac{1}{2}\right)^\frac{p}{p-1}-\left|t-\frac{1}{2}\right|^\frac{p}{p-1}\right),\qquad \mbox{ where } C\in\mathbb{R}\setminus\{0\}.
\]
We refer to Appendix \ref{sec:A} and to \cite[Section 5]{FL} for more details on the constant $\pi_{p,q}$.
\end{oss}

\section{Proof of the Main Theorem}
\label{sec:proof}

We divide the proof in two parts: we first prove the inequality and then discuss the equality cases. We write the proof for the case of a finite $q$: when $p>N$ and $q=+\infty$, the proof has to be suitably modified. We leave the details to the reader.
\vskip.2cm\noindent
{\bf Part 1: inequality}. We take $\varphi$ a $W^{1,p}$ function of one variable, defined on $[0,+\infty)$ and such that $\varphi(0)=0$. We insert the test function
\[
v=\varphi \circ \delta_\Omega,
\]
in the Rayleigh quotient defining $\lambda_{p,q}(\Omega)$. For every $\tau\in[0,R_\Omega]$, we indicate
\[
\Omega_\tau=\{x\in\Omega\, :\, \delta_\Omega(x)>\tau\}.
\] 
By using the Coarea Formula, we get
\[
\lambda_{p,q}(\Omega)\le \frac{\displaystyle \int_0^{R_\Omega} |\varphi'(\tau)|^p\,P(\Omega_\tau)\,d\tau}{\displaystyle\left(\int_0^{R_\Omega} |\varphi(\tau)|^q\,P(\Omega_\tau)\,d\tau\right)^\frac{p}{q}}.
\]
We now set $\xi(\tau)=|\Omega_\tau|$
and use the change of variable
\[
s=\xi(\tau)\qquad \mbox{ so that }\qquad ds=\xi'(\tau)\,d\tau=-P(\Omega_{\tau})\,d\tau.
\]
Thus we get
\begin{equation}
\label{stimetta}
\lambda_{p,q}(\Omega)\le \frac{\displaystyle \int_0^{|\Omega|} |\varphi'(\xi^{-1}(s))|^p\,\,ds}{\displaystyle\left(\int_0^{|\Omega|} |\varphi(\xi^{-1}(s))|^q\,ds\right)^\frac{p}{q}}.
\end{equation}
We now take $\psi\in W^{1,p}((0,1))$ such that $\psi\left(1\right)=0$,
to be optimal in the one-dimensional problem 
\[
A_{p,q}:=\min_{u\in W^{1,p}((0,1))\setminus\{0\}}\left\{\frac{\displaystyle\int_0^1|u'|^p\,dt}{\displaystyle\left(\int_0^1|u|^q\,dt\right)^\frac{p}{q}}\, :\, u(0)=0\, \mbox{ or }\,u(1)=0\right\}
\]
Then we make the choice
\[
\varphi(t)=\psi\left(\frac{\xi(t)}{|\Omega|}\right),\qquad t\in\left(0,R_\Omega\right),
\]
in \eqref{stimetta}. This gives
\[
\lambda_{p,q}(\Omega)\le \frac{1}{|\Omega|^p}\,\frac{\displaystyle \int_0^{|\Omega|} \left|\psi'\left(\frac{s}{|\Omega|}\right)\right|^p\,|\xi'(\xi^{-1}(s))|^p\,ds}{\displaystyle\left(\int_0^{|\Omega|} \left|\psi\left(\frac{s}{|\Omega|}\right)\right|^q\,ds\right)^\frac{p}{q}}.
\]
By observing that\footnote{We use here that for a convex set, the function $\tau\mapsto P(\Omega_\tau)$ is monotone decreasing.}
\begin{equation}
\label{perimetri}
|\xi'(\xi^{-1}(s))|=P_{\Omega_{\xi^{-1}(s)}}\le P(\Omega),
\end{equation}
we finally get
\[
\lambda_{p,q}(\Omega)\le \frac{\displaystyle \int_0^{|\Omega|}\left|\psi'\left(\frac{s}{|\Omega|}\right)\right|^p\,ds}{\displaystyle\left(\int_0^{|\Omega|}\left|\psi\left(\frac{s}{|\Omega|}\right)\right|^q\,ds\right)^\frac{p}{q}}\,\left(\frac{P(\Omega)}{|\Omega|}\right)^p=A_{p,q}\,\frac{P(\Omega)^p}{|\Omega|^{p-1+\frac{p}{q}}}.
\]
By using Lemma \ref{lm:1D}, we conclude the proof.
\vskip.2cm\noindent
{\bf Part 2: sharpness}. We first observe that inequality is strict for every open bounded convex set $\Omega\subset\mathbb{R}^N$. Indeed, if for some $\Omega$ it would hold
\[
\lambda_{p,q}(\Omega)=\left(\frac{\pi_{p,q}}{2}\right)^p\,\left(\frac{P(\Omega)}{|\Omega|^{1-\frac{1}{p}+\frac{1}{q}}}\right)^p,
\]
in particular we would get equality in \eqref{perimetri}, for almost every $s\in(0,|\Omega|)$. This is clearly not possible for a bounded set.
\par
We now prove that the estimate is sharp for $1\le q\le p$. For every $L>0$, we take the convex set
\[
\Omega_L=\left(-\frac{L}{2},\frac{L}{2}\right)^{N-1}\times (0,1).
\]
We take $u_L\in \mathcal{D}^{1,p}_0(\Omega_L)\setminus\{0\}$ to be optimal for $\lambda_{p,q}(\Omega_L)$, i.e.
\[
\lambda_{p,q}(\Omega_L)=\frac{\displaystyle\int_{\Omega_L}|\nabla u_L|^p\,dx}{\displaystyle\left( \int_{\Omega_L} |u_L|^q\,dx\right)^\frac{p}{q}}\ge \frac{\displaystyle\int_{[-L/2,L/2]^{N-1}}\left(\int_0^1 |(u_L)_{x_N}|^p\,dx_N\right)\,dx'}{\displaystyle\left( \int_{[-L/2,L/2]^{N-1}} \left(\int_0^1 |u_L|^q\,dx_N\right)\,dx'\right)^\frac{p}{q}}.
\]
By performing the change of variable $x'=L\,y'$ with $y'\in [-1/2,1/2]^{N-1}$, we get
\[
\lambda_{p,q}(\Omega_L)\ge \frac{1}{L^{(N-1)\,\left(\frac{p}{q}-1\right)}}\, \frac{\displaystyle\int_{[-1/2,1/2]^{N-1}}\left(\int_0^1 |(u_L(L\,y',x_N))_{x_N}|^p\,dx_N\right)\,dy'}{\displaystyle\left( \int_{[-1/2,1/2]^{N-1}} \left(\int_0^1 |u_L(L\,y',x_N)|^q\,dx_N\right)\,dy'\right)^\frac{p}{q}}.
\]
We now observe that for almost every $y'\in [-1/2,1/2]^{N-1}$, the function $x_N\mapsto u_L(L\,y',x_N)$ is admissible in the variational problem
\[
\min_{u\in W^{1,p}((0,1))\setminus\{0\}}\left\{\frac{\displaystyle\int_{0}^1|u'|^p\,dt}{\displaystyle\left(\int_{0}^1|u|^q\,dt\right)^\frac{p}{q}}\, :\, u(0)=u(1)=0\right\}=\Big(\pi_{p,q}\Big)^p.
\] 
Thus, we get
\[
\lambda_{p,q}(\Omega_L)\ge \frac{\Big(\pi_{p,q}\Big)^p}{L^{(N-1)\,\left(\frac{p}{q}-1\right)}}\, \frac{\displaystyle\int_{[-1/2,1/2]^{N-1}}\left(\int_0^1 |u_L(L\,y',x_N)|^q\,dx_N\right)^\frac{p}{q}\,dy'}{\displaystyle\left( \int_{[-1/2,1/2]^{N-1}} \left(\int_0^1 |u_L(L\,y',x_N)|^q\,dx_N\right)\,dy'\right)^\frac{p}{q}}.
\]
By observing that we have $p/q\ge 1$, Jensen's inequality finally implies
\begin{equation}
\label{hey!}
\lambda_{p,q}(\Omega_L)\ge \frac{\Big(\pi_{p,q}\Big)^p}{L^{(N-1)\,\left(\frac{p}{q}-1\right)}}.
\end{equation}
By direct computation, we get
\[
P(\Omega_L)=2\,(N-1)\,L^{N-2}+2\,L^{N-1}\qquad \mbox{ and }\qquad |\Omega_L|=L^{N-1},
\]
that is
\begin{equation}
\label{rapportino}
\left(\frac{P(\Omega_L)}{|\Omega_L|^{1-\frac{1}{p}+\frac{1}{q}}}\right)^p\sim \frac{2^p}{L^{(N-1)\,\left(
\frac{p}{q}-1\right)}},\qquad \mbox{ for }L\to +\infty.
\end{equation}
By using this information in \eqref{hey!}, we obtain
\begin{equation}
\label{egolo!}
\lambda_{p,q}(\Omega_L)\ge \frac{\Big(\pi_{p,q}\Big)^p}{L^{(N-1)\,\left(\frac{p}{q}-1\right)}}\sim \left(\frac{\pi_{p,q}}{2}\right)^p\,\left(\frac{P(\Omega_L)}{|\Omega_L|^{1-\frac{1}{p}+\frac{1}{q}}}\right)^p,\qquad \mbox{ for }L\to+\infty.
\end{equation}
On the other hand, by the first part of the proof we also have
\[
\lambda_{p,q}(\Omega_L)< \left(\frac{\pi_{p,q}}{2}\right)^p\,\left(\frac{P(\Omega_L)}{|\Omega_L|^{1-\frac{1}{p}+\frac{1}{q}}}\right)^p.
\]
This and \eqref{egolo!} finally give
\begin{equation}
\label{tuttassieme}
\lim_{L\to+\infty}\lambda_{p,q}(\Omega_L)\,\left(\frac{|\Omega_L|^{1-\frac{1}{p}+\frac{1}{q}}}{P(\Omega_L)}\right)^p=\left(\frac{\pi_{p,q}}{2}\right)^p,
\end{equation}
which proves the optimality of the estimate for $q\le p$. $\hfill\square$

\begin{oss}[The case $q>p$]
The above computations badly fail to show sharpness in the case $q>p$. Indeed, in this case by \eqref{rapportino}
\[
\left(\frac{P(\Omega_L)}{|\Omega_L|^{1-\frac{1}{p}+\frac{1}{q}}}\right)^p\sim \frac{2^p}{L^{(N-1)\,\left(
\frac{p}{q}-1\right)}},\qquad \mbox{ as } L\to +\infty,
\]
and the last quantity diverges to $+\infty$, thanks to the fact that now $p/q-1<0$. On the other hand, it is not difficult to show that
\[
\lim_{L\to+\infty} \lambda_{p,q}(\Omega_L)=\lambda_{p,q}(\mathbb{R}^{N-1}\times (0,1))<+\infty.
\]
We will show in Section \ref{sec:pq} that there is a deep reason behind this failure.
\end{oss}

\begin{oss}[More general sets]
The crucial ingredient of the proof of the Main Theorem is the fact that 
\[ 
\tau\mapsto P(\Omega_\tau) \mbox{ is monotone decreasing} \qquad \mbox{마nd }\qquad \lim_{t\to 0^+} P(\Omega_t)=P(\Omega). 
\]
Thus the same P\'olya--type estimate on $\lambda_{p,q}$ can be obtained for sets having such a property. This happens for example if $\Omega\subset\mathbb{R}^N$ is an open bounded set with $C^1$ boundary, such that the distance function $\delta_\Omega$ is {\it weakly superharmonic in $\Omega$}, i.e. if
\begin{equation}
\label{distance}
\int_\Omega \langle \nabla \delta_\Omega,\nabla \varphi\rangle\,dx\ge 0,
\end{equation}
for every non-negative $\varphi\in W^{1,2}(\Omega)$ with compact support in $\Omega$. Indeed, by taking a test function of the form
\[
\varphi(x)=\psi(\delta_\Omega(x)),\qquad \psi\ge 0,
\]
and use Coarea Formula, from \eqref{distance} we get
\[
0\le \int_\Omega |\nabla \delta_\Omega|^2\,\psi'(\delta_\Omega)\,dx=\int_0^{R_\Omega} \psi'(\tau)\,P(\Omega_\tau)\,d\tau.
\]
We now choose $0<s_0<s_1<R_\Omega$ and take for $\varepsilon\ll 1$
\[
\psi_\varepsilon(t)=\left\{\begin{array}{rl}
0, & \mbox{ if }0\le t\le s_0-\varepsilon,\\
\dfrac{1}{2\,\varepsilon}\,(t-s_0+\varepsilon),& \mbox{ if } s_0-\varepsilon<t<s_0+\varepsilon,\\
1& \mbox{ if } s_0+\varepsilon\le t\le s_1-\varepsilon,\\
\dfrac{1}{2\,\varepsilon}\,(s_1+\varepsilon-t),& \mbox{ if } s_1-\varepsilon<t<s_1+\varepsilon,\\
0, & \mbox{ if } s_1+\varepsilon\le t\le 1.
\end{array}
\right.
\]
which is just a smooth approximation of the characteristic function $1_{[s_0,s_1]}$. By using the formula above, we get
\[
0\le \frac{1}{2\,\varepsilon}\,\int_{s_0-\varepsilon}^{s_0+\varepsilon} P(\Omega_\tau)\,d\tau-\frac{1}{2\,\varepsilon}\,\int_{s_1-\varepsilon}^{s_1+\varepsilon} P(\Omega_\tau)\,d\tau.
\]
As a consequence of Coarea Formula, the function $t\mapsto P(\Omega_t)$ is $L^1$. Thus by taking $s_0$ and $s_1$ to be two Lebesgue points of this function, we can pass to the limit  as $\varepsilon$ goes to $0$ and get
\[
P(\Omega_{s_1})\le P(\Omega_{s_0}).
\]
The fact that $P(\Omega_t)\to P(\Omega)$ as $t$ goes to $0$, follows by the Area Formula and the smoothness of the boundary.
\par
We recall that in dimension $N=2$, condition \eqref{distance} implies that $\Omega$ has to be convex (see \cite[Theorem 2]{AK}), but for $N\ge 3$ this is a more general condition (see \cite[Section 5]{AK}).
\end{oss}

\section{Buser's inequality for convex sets}
\label{sec:consequence}

This was the original motivation of the present paper.
\begin{coro}[Buser's inequality for the $p-$Laplacian]
Let $1<p<+\infty$ and let $\Omega\subset \mathbb{R}^N$ be an open convex set. Then we have
\begin{equation}
\label{buser}
\lambda_p(\Omega)\le \left(\frac{\pi_p}{2}\right)^p\,\Big(h_1(\Omega)\Big)^p,
\end{equation}
and the inequality is strict on bounded convex sets. Moreover, the constant $(\pi_p/2)^p$ is sharp and the equality sign is attained by a {\rm slab}, i.e. any set of the form
\[
\{x\in\mathbb{R}^N\, :\, a<\langle x,\omega\rangle<b\},\qquad \mbox{ for some } a,b\in\mathbb{R} \mbox{ and }\omega\in\mathbb{S}^{N-1}.
\]
\end{coro}
\begin{proof}
We first prove \eqref{buser} for bounded sets, then we extend it to general sets.
\vskip.2cm\noindent 
{\bf Bounded convex sets}. Let $E_\Omega\subset \Omega$ be a Cheeger set for $\Omega$. By \cite[Theorem 1]{AC}, this is convex as well. By monotonicity of $\lambda_p$, we have
\[
\lambda_p(\Omega)\le \lambda_p(E_\Omega).
\]
By applying the Main Theorem and using that $E_\Omega$ is a Cheeger set, we obtain
\[
\lambda_p(\Omega)\le \lambda_p(E_\Omega)<\left(\frac{\pi_p}{2}\right)^p\,\left(\frac{P(E_\Omega)}{|E_\Omega|}\right)^p=\left(\frac{\pi_p}{2}\right)^p\,\Big(h_1(\Omega)\Big)^p,
\]
as desired. 
\vskip.2cm\noindent
{\bf General convex sets.} For every $R>0$ sufficiently large, we set $\Omega_R=\Omega\cap B_R(0)$. From the previous step we have
\[
\lambda_p(\Omega)\le \lambda_p(\Omega_R)\le \left(\frac{\pi_p}{2}\right)^p\,\Big(h_1(\Omega_R)\Big)^p.
\] 
We used again that $\lambda_p$ is monotone decreasing with respect to set inclusion. We now fix $\varepsilon>0$ and take a bounded set $E\subset \Omega$ such that $|E|>0$ and 
\[
\frac{P(E)}{|E|}<h_1(\Omega)+\varepsilon.
\]
By taking $R>0$ large enough, we have $E\subset \Omega_R$ as well, thus in particular
\[
h_1(\Omega_R)\le \frac{P(E)}{|E|}<h_1(\Omega)+\varepsilon.
\]
This implies 
\[
\lambda_p(\Omega)< \left(\frac{\pi_p}{2}\right)^p\,\Big(h_1(\Omega)+\varepsilon\Big)^p,
\]
and thus the conclusion, by the arbitrariness of $\varepsilon>0$. In order to show that we have equality for slabs, we take the set
\[
\Omega=\mathbb{R}^{N-1}\times(0,1),
\]
and again the sequence
\[
\Omega_L=\left(-\frac{L}{2},\frac{L}{2}\right)^{N-1}\times(0,1).
\]
By using $\Omega_L$ as an admissible set for $h_1(\Omega)$ and recalling \eqref{rapportino}, we have
\[
\Big(h_1(\mathbb{R}^{N-1}\times(0,1))\Big)^p\le \lim_{L\to+\infty}\left(\frac{P(\Omega_L)}{|\Omega_L|}\right)^p= 2^p.
\]
On the other hand, by Lemma \ref{lm:slab} we have
\[
\lambda_p(\mathbb{R}^{N-1}\times(0,1))=\Big(\pi_p\Big)^p.
\]
We thus obtain
\[
\Big(\pi_p\Big)^p=\lambda_p(\mathbb{R}^{N-1}\times(0,1))\le \left(\frac{\pi_p}{2}\right)^p\,\Big(h_1(\mathbb{R}^{N-1}\times(0,1))\Big)^p\le 2^p\,\Big(\frac{\pi_p}{2}\Big)^p,
\]
which gives the desired result.
\end{proof}
\begin{oss}
The previous result generalizes \cite[Proposition 4.1]{Pa} to $p\not =2$ and to every dimension $N\ge 2$.
\end{oss}
\begin{oss}
We do not claim that convexity is necessary for Theorem \ref{buser} to hold. However, it is easy to see that inequality \eqref{buser} can not hold for general sets. This is due to the fact that if we remove from $\Omega$ a set with zero $N-$dimensional Lebesgue measure and positive $p-$capacity (for example, a compact Lipschitz $(N-1)-$dimensional surface),  $h_1$ is unchanged while $\lambda_p$ increases. By iterating this construction, we can easily produce a sequence $\{\Omega_n\}_{n\in\mathbb{N}}\subset\mathbb{R}^N$ of open bounded sets, such that
\[
\lim_{n\to\infty} \lambda_p(\Omega_n)=+\infty\qquad \mbox{ and }\qquad h_1(\Omega_n)\le C.
\]
\end{oss}

\section{A closer look at the case $q>p$}
\label{sec:pq}

In this section, we will show that for $q>p$ the estimate of the Main Theorem is not optimal. Indeed, we are going to prove the appearance of a weird phenomenon: for $q>p$, the scale invariant quantity
\[
\lambda_{p,q}(\Omega)\, \left(\frac{|\Omega|^{1-\frac{1}{p}+\frac{1}{q}}}{P(\Omega)}\right)^p,
\]
admits a maximizer in the class of open bounded convex sets.
\vskip.2cm
\begin{defi}
Let $\{\Omega_n\}_{n\in\mathbb{N}}\subset\mathbb{R}^N$ and $\Omega\subset\mathbb{R}^N$ be open sets, such that
\[
\Omega\subset \overline{B_R(0)}\qquad \mbox{ and }\qquad \Omega_n\subset \overline{B_{R}(0)}, \quad \mbox{ for every } n\in\mathbb{N},
\]
for some $R>0$. We say that $\{\Omega_n\}_{n\in\mathbb{N}}$ {\it converges to $\Omega$ with respect to the Hausdorff complementary metric} if
\[
\lim_{n\to\infty}d_H(\overline{B_R(0)}\setminus\Omega_n,\overline{B_R(0)}\setminus\Omega)=0,
\]
where $d_H$ is the Hausdorff distance. We use the notation $\Omega_n\stackrel{H}{\longrightarrow} \Omega$.
\end{defi}
We start with a simple technical result. Its proof is standard routine, we include it for completeness.
\begin{lm}
\label{lm:usc}
Let $1<p<+\infty$ and 
\[
\left\{\begin{array}{cc}
1\le q<p^*,& \mbox{ if } p\le N,\\
&\\
1\le q\le +\infty, & \mbox{ if }p>N.
\end{array}
\right.
\] 
Let $\{\Omega_n\}_{n\in\mathbb{N}}\subset \mathbb{R}^N$ be a sequence of open sets, such that 
\[
\Omega_n\subset \overline{B_{R}(0)}, \quad \mbox{ for every } n\in\mathbb{N}\qquad \mbox{ and }\qquad \Omega_n\stackrel{H}{\longrightarrow} \Omega,
\]
where $\Omega\subset\mathbb{R}^N$ is an open bounded set.
Then we have
\[
\limsup_{n\to\infty} \lambda_{p,q}(\Omega_n)\le \lambda_{p,q}(\Omega).
\]
\end{lm}
\begin{proof}
Let $\varphi\in C^\infty_0(\Omega)\setminus\{0\}$, by \cite[Proposition 2.2.15]{HP} we have that $\varphi\in C^\infty_0(\Omega_n)$ for $n$ large enough. This implies
\[
\limsup_{n\to\infty}\lambda_{p,q}(\Omega_n)\le \limsup_{n\to\infty} \frac{\displaystyle\int_{\Omega_n}|\nabla \varphi|^p\,dx}{\displaystyle\left(\int_{\Omega_n} |\varphi|^q\,dx\right)^\frac{p}{q}}=\frac{\displaystyle\int_{\Omega}|\nabla \varphi|^p\,dx}{\displaystyle\left(\int_{\Omega} |\varphi|^q\,dx\right)^\frac{p}{q}}.
\]
By arbitrariness of $\varphi\in C^\infty_0(\Omega)$, we get the desired conclusion.
\end{proof}

\begin{teo}
\label{teo:existence}
Let $1<p<+\infty$ and 
\[
\left\{\begin{array}{cc}
p<q<p^*,& \mbox{ if } p\le N,\\
&\\
p< q\le +\infty, & \mbox{ if }p>N.
\end{array}
\right.
\] 
Then the shape optimization problem
\[
\sup\left\{\left(\frac{|\Omega|^{1-\frac{1}{p}+\frac{1}{q}}}{P(\Omega)}\right)^p\,\lambda_{p,q}(\Omega)\, :\, \Omega\subset\mathbb{R}^N \mbox{ open bounded convex set}\right\},
\]
admits a solution $\Omega^*$. In particular, we have the following scale invariant estimate
\[
\left(\frac{|\Omega|^{1-\frac{1}{p}+\frac{1}{q}}}{P(\Omega)}\right)^p\,\lambda_{p,q}(\Omega)\le \left(\frac{|\Omega^*|^{1-\frac{1}{p}+\frac{1}{q}}}{P(\Omega^*)}\right)^p\,\lambda_{p,q}(\Omega^*),
\]
for every $\Omega\subset\mathbb{R}^N$ open bounded convex set.
\end{teo}
\begin{proof}
We use the Direct Methods in the Calculus of Variations. We call $\lambda^*_{p,q}$ the supremum above, then this is not $+\infty$. Indeed, by the Main Theorem we have
\[
\left(\frac{|\Omega|^{1-\frac{1}{p}+\frac{1}{q}}}{P(\Omega)}\right)^p\,\lambda_{p,q}(\Omega)<\left(\frac{\pi_{p,q}}{2}\right)^p,
\]
for every admissible convex set $\Omega$. Of course, we also have $\lambda^*_{p,q}>0$. We then take a sequence of admissible sets $\{\Omega_n\}_{n\in\mathbb{N}}$ such that
\[
\left(\frac{|\Omega_n|^{1-\frac{1}{p}+\frac{1}{q}}}{P(\Omega_n)}\right)^p\,\lambda_{p,q}(\Omega_n)>\frac{n+1}{n+2}\,\lambda^*_{p,q},\qquad \mbox{ for every }n\in\mathbb{N}.
\] 
As the functional we are optimizing is scale invariant, we can assume without loss of generality that
\[
\frac{|\Omega_n|^{1-\frac{1}{p}+\frac{1}{q}}}{P(\Omega_n)}=1,\qquad \mbox{ for every }n\in\mathbb{N}.
\]
Observe that this is possible, also thanks to the fact that 
\[
1-\frac{1}{p}+\frac{1}{q}>\frac{N-1}{N}.
\]
Then by using the {\it isoperimetric inequality}, we get
\[
1=\frac{|\Omega_n|^{1-\frac{1}{p}+\frac{1}{q}}}{P(\Omega_n)}\le \frac{1}{N\,\omega_N^{1/N}}\,|\Omega_n|^{\frac{1}{q}-\frac{1}p+\frac{1}{N}}.
\]
The assumption on $q$ entails that the last exponent is positive, thus from the previous estimate we get the uniform lower bound
\begin{equation}
\label{dalbasso}
|\Omega_n|\ge (N\,\omega_N)^\frac{1}{\frac{1}{N}-\frac{1}{p}+\frac{1}{q}}>0\,\qquad \mbox{ for every }n\in\mathbb{N}.
\end{equation} 
We now observe that by the monotonicity and scaling properties of $\lambda_{p,q}$, we get
\[
\frac{n+1}{n+2}\,\lambda^*_{p,q}<\lambda_{p,q}(\Omega_n)\le \lambda_{p,q}(B_1(0))\, R_{\Omega_n}^{N-p-\frac{p}{q}\,N}.
\]
By observing that the last exponent is negative, we get that $R_{\Omega_n}$ is uniformly bounded from above. We can now apply inequality \eqref{good} of Proposition \ref{prop:inradius} with\footnote{It is precisely here that the hypothesis $q>p$ is needed: indeed
\[
\alpha<1\qquad \Longleftrightarrow \qquad q>p.
\]}
\[
\alpha=1-\frac{1}{p}+\frac{1}{q},
\]
and obtain that
\begin{equation}
\label{diametri}
\mathrm{diam}(\Omega_n)\le C,\qquad \mbox{ for every } n\in\mathbb{N}.
\end{equation} 
This property implies that we can assume 
\[
\Omega_n\subset \overline{B_R(0)},\qquad \mbox{ for every }n\in\mathbb{N},
\]
up to a translation and up to take $R>0$ large enough. 
\par
We can now use the Blaschke Selection Theorem (see \cite[Theorem 1.8.7]{Sch}), so to get that $\{\Omega_n\}_{n\in\mathbb{N}}$ converges
 (up to subsequences) with respect to the Hausdorff complementary metric to a limit open set $\Omega^*\subset\overline{B_R(0)}$, which is still convex. By observing that $\lambda_{p,q}$ is upper semicontinuous (see Lemma \ref{lm:usc}), we have
\[
\lambda^*_{p,q}\le \limsup_{n\to\infty} \lambda_{p,q}(\Omega_n)\le\lambda_{p,q}(\Omega^*). 
\]
Moreover, the uniform bound \eqref{diametri} and the monotonicity of the perimeter with respect to inclusion for convex sets, implies that $\{1_{\Omega_n}\}_{n\in\mathbb{N}}$ is a bounded sequence in the space of functions with bounded variation $BV(\mathbb{R}^N)$. Thus we get
\[
\lim_{n\to\infty} \|1_{\Omega_n}-1_{\Omega^*}\|_{L^1(\mathbb{R}^N)}=0\qquad \mbox{ and }\qquad P(\Omega^*)\le \liminf_{n\to\infty} P(\Omega_n).
\]
The two properties and \eqref{dalbasso} imply that
\[
|\Omega^*|>0\qquad \mbox{ and }\qquad \frac{|\Omega^*|^{1-\frac{1}{p}+\frac{1}{q}}}{P(\Omega^*)}\ge 1,
\]
thus we have
\[
\lambda^*_{p,q}\le \lambda_{p,q}(\Omega^*)\le \left(\frac{|\Omega^*|^{1-\frac{1}{p}+\frac{1}{q}}}{P(\Omega^*)}\right)^p\,\lambda_{p,q}(\Omega^*)\le \lambda^*_{p,q}.
\] 
Thus $\Omega^*$ is the desired maximizer.
\end{proof}
We now discuss the behaviour of the scale invariant quantity
\[
\lambda_{p,q}(\Omega)\, \left(\frac{|\Omega|^{1-\frac{1}{p}+\frac{1}{q}}}{P(\Omega)}\right)^p
\] 
as $q\nearrow p^*$, when $p\le N$.
\begin{oss}[Limit case]
\label{oss:palle}
In the subconformal case $p<N$, if we define the sharp Sobolev constant
\[
S_{N,p}=\inf_{u\in C^\infty_0(\mathbb{R}^N)\setminus\{0\}} \frac{\displaystyle\int_{\mathbb{R}^N} |\nabla u|^p\,dx}{\displaystyle\left(\int_{\mathbb{R}^N} |u|^\frac{N\,p}{N-p}\,dx\right)^\frac{N}{N-p}},
\]
we have
\[
\lim_{q\nearrow p^*}\lambda_{p,q}(\Omega)=S_{N,p}\qquad \mbox{ and }\qquad \lim_{q\nearrow p^*}\frac{|\Omega|^{1-\frac{1}{p}+\frac{1}{q}}}{P(\Omega)}=\frac{|\Omega|^\frac{N-1}{N}}{P(\Omega)},
\]
and the latter is the classical isoperimetric ratio. Thus we get
\[
\lim_{q\nearrow p^*}\lambda_{p,q}(\Omega)\, \left(\frac{|\Omega|^{1-\frac{1}{p}+\frac{1}{q}}}{P(\Omega)}\right)^p=S_{N,p}\,\left(\frac{|\Omega|^\frac{N-1}{N}}{P(\Omega)}\right)^p,
\]
and the unique maximizers of the last functional are the balls.
\par
The conformal case $p=N$ is slightly different: indeed, in this case we still have
\[
\lim_{q\nearrow +\infty}\frac{|\Omega|^{1-\frac{1}{N}+\frac{1}{q}}}{P(\Omega)}=\frac{|\Omega|^\frac{N-1}{N}}{P(\Omega)},
\]
but the relevant Poincar\'e-Sobolev constant now degenerates, i.e.
\[
\lim_{q\nearrow +\infty} \lambda_{N,q}(\Omega)=0,
\]
see \cite[Lemma 2.2]{RW}. More precisely, \cite[Lemma 2.2]{RW} proves that
\[
\lim_{q\nearrow +\infty} q^{N-1}\,\lambda_{N,q}(\Omega)=C_N,
\]
for a constant $C_N$ depending on $N$ only.
Thus in this case
\[
\lim_{q\nearrow +\infty} q^{N-1}\,\lambda_{N,q}(\Omega)\,\left(\frac{|\Omega|^{1-\frac{1}{N}+\frac{1}{q}}}{P(\Omega)}\right)^N=C_N\,\left(\frac{|\Omega|^\frac{N-1}{N}}{P(\Omega)}\right)^N,
\]
and the latter is again (uniquely) maximized by balls.
\end{oss}
The previous remark suggests that for $p\le N$ and $q$ close to $p^*$, optimizers should look ``round''. One could conjecture that solutions are given by balls. On the other hand, this is surely not the case for $q$ close to $p$, as shown in the following
\begin{prop}
\label{prop:palle}
Let $1<p<+\infty$, there exists $q_0>p$ such that for every $p<q\le q_0$ balls are not solutions of  
\[
\max\left\{\left(\frac{|\Omega|^{1-\frac{1}{p}+\frac{1}{q}}}{P(\Omega)}\right)^p\,\lambda_{p,q}(\Omega)\, :\, \Omega\subset\mathbb{R}^N \mbox{ open bounded convex set}\right\}.
\]
\end{prop}
\begin{proof}
We use a continuity argument, aiming at proving that a set of the form 
\[
\Omega_L=\left(-\frac{L}{2},\frac{L}{2}\right)^{N-1}\times (0,1),
\]
gives a higher value than a ball.
We set for simplicity $B=B_1(0)$, then we take
\[
\varepsilon=\frac{1}{2}\,\left(\left(\frac{\pi_p}{2}\right)^p-\left(\frac{|B|}{P(B)}\right)^p\,\lambda_p(B)\right)>0.
\]
Observe that this is positive, since $B$ can not attain the equality in \eqref{polyapq}. 
We recall that by \eqref{tuttassieme}
\[
\lim_{L\to+\infty} \left(\frac{|\Omega_L|}{P(\Omega_L)}\right)^p\,\lambda_p(\Omega_L)=\left(\frac{\pi_p}{2}\right)^p.
\]
Then there exists $L_\varepsilon>0$ such that 
\[
\left(\frac{|\Omega_{L_\varepsilon}|}{P(\Omega_{L_\varepsilon})}\right)^p\,\lambda_p(\Omega_{L_\varepsilon})>\left(\frac{\pi_p}{2}\right)^p-\varepsilon.
\]
By Lemma \ref{lm:continuity} the quantity $\lambda_{p,q}(\Omega)$ is right continuous in $q$, thus
\[
\lim_{q\searrow p} \left(\frac{|\Omega_{L_\varepsilon}|^{1-\frac{1}{p}+\frac{1}{q}}}{P(\Omega_{L_\varepsilon})}\right)^p\,\lambda_{p,q}(\Omega_{L_\varepsilon})=\left(\frac{|\Omega_{L_\varepsilon}|}{P(\Omega_{L_\varepsilon})}\right)^p\,\lambda_p(\Omega_{L_\varepsilon}).
\]
This implies that there exists $q_\varepsilon>p$ such that for every $p<q\le q_\varepsilon$
\[
\left(\frac{|\Omega_{L_\varepsilon}|^{1-\frac{1}{p}+\frac{1}{q}}}{P(\Omega_{L_\varepsilon})}\right)^p\,\lambda_{p,q}(\Omega_{L_\varepsilon})>\left(\frac{|\Omega_{L_\varepsilon}|}{P(\Omega_{L_\varepsilon})}\right)^p\,\lambda_p(\Omega_{L_\varepsilon})-\varepsilon>\left(\frac{\pi_p}{2}\right)^p-2\,\varepsilon.
\]
We now observe that by Lemma \ref{lm:monotone} we have
\[
\left(\frac{|B|}{P(B)}\right)^p\,\lambda_p(B)\ge \left(\frac{|B|^{1-\frac{1}{p}+\frac{1}{q}}}{P(B)}\right)^p\,\lambda_{p,q}(B),
\]
for every $q>p$.
In conclusion, by recalling the definition of $\varepsilon$ and using the previous estimate, we get
\[
\left(\frac{|\Omega_{L_\varepsilon}|^{1-\frac{1}{p}+\frac{1}{q}}}{P(\Omega_{L_\varepsilon})}\right)^p\,\lambda_{p,q}(\Omega_{L_\varepsilon})>\left(\frac{|B|^{1-\frac{1}{p}+\frac{1}{q}}}{P(B)}\right)^p\,\lambda_{p,q}(B).
\]
This concludes the proof.
\end{proof}

\appendix 

\section{One-dimensional Poincar\'e constants}
\label{sec:A}
In the proof of the Main Theorem, we used the following simple result.
\begin{lm}
\label{lm:1D}
Let $1<p<+\infty$ and $1\le q\le +\infty$, we define
\[
A_{p,q}=\min_{u\in W^{1,p}((0,1))\setminus\{0\}}\left\{\frac{\displaystyle\int_0^1|u'|^p\,dt}{\displaystyle\left(\int_0^1|u|^q\,dt\right)^\frac{p}{q}}\, :\, u(0)=0\, \mbox{ or }\,u(1)=0\right\}
\]
%and
%\[
%B_{p,q}=\min_{u\in W^{1,p}((-1,1))\setminus\{0\}}\left\{\frac{\displaystyle\int_{-1}^1|u'|^p\,dt}{\displaystyle\left(\int_{-1}^1|u|^q\,dt\right)^\frac{p}{q}}\, :\, u(-1)=u(1)=0\right\}.
%\]
then we have
\[
A_{p,q}=\left(\frac{\pi_{p,q}}{2}\right)^p.
\]
\end{lm}
\begin{proof}
We take $u$ to be optimal for $A_{p,q}$, we can assume without loss of generality that
\[
\int_0^1 |u|^q\,dt=1,\qquad u\ge 0\qquad\mbox{ and }\qquad u(1)=0.
\]
We construct the new function $U\in W^{1,p}((-1/2,1/2))$ by
\[
U(t)=\left\{\begin{array}{rl}
u(2\,t),& \mbox{ if }0\le t<1/2,\\
u(-2\,t),& \mbox{ if } -1/2<t<0,
\end{array}
\right.
\]
which is admissible for $\pi_{p,q}$. 
By observing that
\[
\int_{-\frac{1}{2}}^\frac{1}{2} |U'|^p\,dt=2^p\,\int_0^1 |u'|^p\,dt=2^p\,A_{p,q}
\]
and
\[
\left(\int_{-\frac{1}{2}}^\frac{1}{2} |U|^q\,dt\right)^\frac{p}{q}=\left(\int_0^1 |u|^q\,dt\right)^\frac{p}{q}=1,
\]
we get
\[
\Big(\pi_{p,q}\Big)^p\le 2^p\,A_{p,q}.
\]
On the other hand, we take $v\in W^{1,p}((-1/2,1/2))$ to be optimal for $\pi_{p,q}$. Again, without loss of generality we can assume
\[
\int_{-1/2}^{1/2} |v|^q\,dt=1\qquad \mbox{마nd }\qquad v\ge 0.
\]
If we indicate by $s\mapsto \mu(s)=|\{t\in(-1/2,1/2)\, :\, v(t)>s\}|$ the distribution function of $v$, we can define the {\it symmetric decreasing rearrangement} $v^*$ through
\[
\{t\in (-1/2,1/2)\, :\, v^*(t)>s\}=\left(-\frac{\mu(s)}{2},\frac{\mu(s)}{2}\right),\qquad s\ge 0.
\]
By construction, we have
\[
\int_{-\frac{1}{2}}^\frac{1}{2} |v|^q\,dt=\int_{-\frac{1}{2}}^\frac{1}{2} |v^*|^q\,dt,
\]
and by the classical {\it P\'olya-Szeg\H{o} principle}
\[
\int_{-\frac{1}{2}}^\frac{1}{2} |v'|^p\,dt\ge \int_{-\frac{1}{2}}^\frac{1}{2} |(v^*)'|^p\,dt.
\]
We also observe that by construction, we have $v^*(t)=v^*(-t)$, thus 
\[
\begin{split}
\Big(\pi_{p,q}\Big)^p=\frac{\displaystyle\int_{-\frac{1}{2}}^\frac{1}{2}|v'|^p\,dt}{\displaystyle\left(\int_{-\frac{1}{2}}^\frac{1}{2}|v|^q\,dt\right)^\frac{p}{q}}
&\ge \frac{\displaystyle\int_{-\frac{1}{2}}^\frac{1}{2}|(v^*)'|^p\,dt}{\displaystyle\left(\int_{-\frac{1}{2}}^\frac{1}{2}|v^*|^q\,dt\right)^\frac{p}{q}}\\
&=\frac{\displaystyle 2\,\int_{0}^\frac{1}{2}|(v^*)'|^p\,dt}{\displaystyle\left(2\,\int_0^\frac{1}{2}|v^*|^q\,dt\right)^\frac{p}{q}}\ge 2^p\,A_{p,q},
\end{split}
\]
where we used the change of variable $2\,t=s$ in the last estimate.
This concludes the proof.
\end{proof}
%\begin{oss}
%\label{oss:salvi!}
%By scaling and Lemma \ref{lm:1D}, we also get
%\[
%\min_{u\in W^{1,p}((0,1))\setminus\{0\}}\left\{\frac{\displaystyle\int_{0}^1|u'|^p\,dt}{\displaystyle\left(\int_{0}^1|u|^q\,dt\right)^\frac{p}{q}}\, :\, u(0)=u(1)=0\right\}=2^{p-1+\frac{p}{q}}\,B_{p,q}=2^p\,A_{p,q}.
%\]
%\end{oss}
\begin{lm}[Eigenvalue of a slab]
\label{lm:slab}
Let $1<p<+\infty$, then we have
\[
\lambda_p(\mathbb{R}^{N-1}\times(0,1))=\min_{u\in W^{1,p}((0,1))\setminus\{0\}}\left\{\frac{\displaystyle\int_{0}^1|u'|^p\,dt}{\displaystyle \int_{0}^1|u|^p\,dt}\, :\, u(0)=u(1)=0\right\}=\Big(\pi_p\Big)^p.
\]
\end{lm}
\begin{proof}
We first prove the upper bound
\begin{equation}
\label{sopra1d}
\lambda_p(\mathbb{R}^{N-1}\times(0,1))\le \Big(\pi_p\Big)^p.
\end{equation}
For every $\varepsilon>0$, we take $u_\varepsilon\in C^\infty_0((0,1))$ to be an almost optimal function for the one-dimensional problem, i.e.
\[
\int_0^1 |u'_\varepsilon|^p\,dt<\Big(\pi_p\Big)^p+\varepsilon\qquad \mbox{ and }\qquad \int_0^1 |u_\varepsilon|^p\,dt=1.
\]
We take $\eta\in C^\infty_0(\mathbb{R})$  such that
\[
0\le \eta\le 1,\qquad \eta\equiv 1 \mbox{ on } \left[-\frac{1}{2},\frac{1}{2}\right],\qquad \eta\equiv 0 \mbox{ on } \mathbb{R}\setminus[-1,1],
\]
then for every $R>0$, we choose 
\[
\varphi(x',x_N)=\eta_R(|x'|)\,u_\varepsilon(x_N),\qquad \mbox{ where }\ \eta_R(t)=R^{\frac{1-N}{p}}\,\eta\left(\frac{t}{R}\right).
\]
We obtain
\[
\lambda_p(\mathbb{R}^{N-1}\times(0,1))\le \frac{\displaystyle \int_{B_{R}(0)}\int_0^1\left(\left|\nabla _{x'}\eta_R\left(|x'|\right)\right|^2\,|u_\varepsilon(x_N)|^2+|u'_\varepsilon(x_N)|^2\,\eta_R\left(|x'|\right)^2\right)^\frac{p}{2}\,dx'\,dx_N}{\displaystyle \int_{B_R(0)} \eta_R\left(|x'|\right)^p\,dx'}.
\]
We now use the definition of $\eta_R$ and the change of variables $x'=R\,y'$, so to get
\[
\begin{split}
\lambda_p(\mathbb{R}^{N-1}\times(0,1))&\le \frac{\displaystyle \int_{B_1(0)}\int_0^1 \left[R^{\frac{2}{p}\,(1-N)-2}\,\left|\eta'(|y'|)\right|^2\,|u_\varepsilon(x_N)|^2+R^{\frac{2}{p}\,(1-N)}\,|u'_\varepsilon(x_N)|^2\,|\eta(|y'|)|^2\right]^\frac{p}{2}R^{N-1}\,dy'\,dx_N}{\displaystyle\int_{B_1(0)}\,|\eta(|y'|)|^p\,dy'}\\
&=\frac{\displaystyle \int_{B_1(0)}\int_0^1 \left[\frac{1}{R^2}\,\left|\eta'(|y'|)\right|^2\,|u_\varepsilon(x_N)|^2+|u'_\varepsilon(x_N)|^2\,|\eta(|y'|)|^2\right]^\frac{p}{2}\,dy'\,dx_N}{\displaystyle\int_{B_1(0)}\,|\eta(|y'|)|^p\,dy'}.
\end{split}
\]
By taking the limit as $R$ goes to $+\infty$ and using Fubini's theorem, from the previous estimate we get
\[
\lambda_p(\mathbb{R}^{N-1}\times(0,1))\le \frac{\displaystyle \int_{B_1(0)}\int_0^1 |u'_\varepsilon(x_N)|^p\,|\eta(|y'|)|^p\,dy'\,dx_N}{\displaystyle\int_{B_1(0)}\,|\eta(|y'|)|^p\,dy'}=\int_0^1 |u'_\varepsilon|^p\,dx_N<\Big(\pi_p\Big)^p+\varepsilon.
\]
Finally, the arbitrariness of $\varepsilon>0$ implies \eqref{sopra1d}.
\vskip.2cm\noindent
We now prove the reverse inequality 
\begin{equation}
\label{sotto1d}
\lambda_p(\mathbb{R}^{N-1}\times (0,1))\ge \Big(\pi_p\Big)^p.
\end{equation}
For every $\varepsilon>0$, we take $\varphi_\varepsilon\in C^\infty_0(\mathbb{R}^{N-1}\times (0,1))\setminus\{0\}$ such that
\[
\frac{\displaystyle\int_{\mathbb{R}^{N-1}\times(0,1)} |\nabla \varphi_\varepsilon|^p\,dx}{\displaystyle\int_{\mathbb{R}^{N-1}\times(0,1)} |\varphi_\varepsilon|^p\,dx}<\lambda_p(\mathbb{R}^{N-1}\times(0,1))+\varepsilon.
\]
Observe that
\[
\begin{split}
\int_{\mathbb{R}^{N-1}\times(0,1)} |\nabla \varphi_\varepsilon|^p\,dx&\ge \int_{\mathbb{R}^{N-1}}\left(\int_0^1 |(\varphi_\varepsilon)_{x_N}|^p\,d x_N\right)\,dx'\\
&\ge \Big(\pi_p\Big)^p\,\int_{\mathbb{R}^{N-1}}\left(\int_0^1 |\varphi_\varepsilon|^p\,d x_N\right)\,dx'=\Big(\pi_p\Big)^p\,\int_{\mathbb{R}^{N-1}\times [0,1]}|\varphi_\varepsilon|^p\,dx,
\end{split}
\]
where we used that $x_N\mapsto \varphi_\varepsilon(x',x_N)$ is admissible for the one-dimensional problem, for every $x'$. We thus obtained
\[
\Big(\pi_p\Big)^p\le \lambda_p(\mathbb{R}^{N-1}\times(0,1))+\varepsilon.
\]
By arbitrariness of $\varepsilon>0$, this proves \eqref{sotto1d}.
\end{proof}

\section{Geometric estimates for convex sets}
\label{sec:B}

We start with the following technical result. 
\begin{lm}
\label{lm:makai}
Let $\Omega\subset\mathbb{R}^N$ be an open bounded convex set. Then we have
\[
\frac{R_\Omega}{N}\le \frac{|\Omega|}{P(\Omega)}< R_\Omega.
\]
Both inequalities are sharp.
\end{lm}
\begin{proof}
The upper bound simply follows from the Coarea Formula applied with the distance function $\delta_\Omega$. By still using the notation
\[
\Omega_\tau=\{x\in\Omega\, :\, \delta_\Omega(x)>\tau\},
\]
we have
\[
|\Omega|=\int_0^{R_\Omega} P(\Omega_\tau)\,d\tau\le R_\Omega\,P(\Omega),
\]
where we used that
\[
R_\Omega=\sup_{\Omega} \delta_\Omega,
\]
and that $t\mapsto P(\Omega_\tau)$ is monotone decreasing, thanks to the convexity of $\Omega$.
\vskip.2cm\noindent
As for the lower bound, let $x_0\in \Omega$ be such that $B_{R_\Omega}(x_0)\subset\Omega$. By the Divergence Theorem, we get
\[
|\Omega|=\frac{1}{N}\,\int_\Omega \mathrm{div}(x-x_0)\,dx=\frac{1}{N}\,\int_{\partial \Omega} \langle x-x_0,\nu(x)\rangle\,d\mathcal{H}^{N-1}.
\]
We now observe that for every $x\in\partial \Omega$, by convexity
\[
\overline{B_{R_\Omega}(x_0)}\subset \overline{\Omega} \subset \Big\{y\in\mathbb{R}^N\, :\, \langle y-x_0,\nu(x)\rangle\le \langle x-x_0,\nu(x)\rangle \Big\}.
\]
In particular, by taking the point $x_0+R_\Omega\,\nu(x)\in\partial B_{R_\Omega}(x_0)$, we get
\[
R_\Omega=\langle (x_0+R_\Omega\,\nu(x))-x_0,\nu(x)\rangle\le \langle x-x_0,\nu(x)\rangle.
\]
By inserting this above, we get
\[
|\Omega| \ge \frac{R_\Omega}{N}\,\int_{\partial \Omega} d\mathcal{H}^{N-1},
\]
which concludes the proof.
\end{proof}
\begin{oss}[Sharpness of the upper bound]
The inequality 
\[
\frac{|\Omega|}{P(\Omega)}< R_\Omega,
\]
is strict and we asymptotically have equality on the slab--type sequence
\[
\Omega_L=\left(-\frac{L}{2},\frac{L}{2}\right)^{N-1}\times\left(0,1\right).
\]
In other words, we have
\[
\lim_{L\to+\infty}\frac{|\Omega_L|}{R_{\Omega_L}\,P(\Omega_L)}=1.
\]
\end{oss}
\begin{oss}[Sharpness of the lower bound]
Here the identification of equality cases is quite subtle. A first family of sets giving the equality in
\[
\frac{R_\Omega}{N}\le \frac{|\Omega|}{P(\Omega)},
\]
is obviously given by balls. Indeed, if $\Omega$ is a ball with radius $R$, we have
\[
|\Omega|=\omega_N\,R^N,\qquad P(\Omega)=N\,\omega_N\,R^{N-1},\qquad R_\Omega=R.
\]
\par
However, by inspecting the proof, another family of sets naturally leads to equality: this is given by rotationally symmetric cones with shrinking opening. More precisely, for every $0<\alpha\ll 1$ let us define
\[
C_\alpha=\{x\in\mathbb{R}^N\, :\, |x|<1\,\mbox{ and }\, x_N>|x|\,\cos\alpha\}.
\]
Then we have
\[
|C_\alpha|=\frac{\omega_{N-1}}{N}\, (\tan\alpha)^{N-1}\,(\cos\alpha)^N+\int_{\cos\alpha}^1 \omega_{N-1}\,(1-t^2)^\frac{N-1}{2}\,dt\sim \frac{\omega_{N-1}}{N}\,\alpha^{N-1},
\] 
\[
P(C_\alpha)=N\,|C_\alpha|+\omega_{N-1}\,(\sin\alpha)^{N-2}\sim \omega_{N-1}\,\alpha^{N-2},
\]
and
\[
R_{C_\alpha}=\frac{\sin\alpha}{\sin\alpha+1}\sim \alpha.
\]
This implies that
\[
\lim_{\alpha\to 0^+}\frac{|C_\alpha|}{R_{C_\alpha}\, P(C_{\alpha})}=\frac{1}{N}.
\]
\end{oss}
\begin{prop}
\label{lm:geometrico}
Let $N\ge 2$, for every $\Omega\subset\mathbb{R}^N$ open bounded convex set, we have
\begin{equation}
\label{ue}
\gamma_N\,\mathrm{diam}(\Omega)\le \frac{P(\Omega)}{R_\Omega^{N-2}},
\end{equation}
where the dimensional constant $\gamma_N$ is given by
\[
\gamma_N=\frac{1}{2}\,\min\left\{\omega_{N-1}\,\left(\frac{3}{4}\right)^\frac{N-1}{2},\,\frac{N\,\omega_N}{2}\right\}.
\]
In particular, we also get
\begin{equation}
\label{ueil}
\gamma_N\,\mathrm{diam}(\Omega)\le \left(\frac{P(\Omega)}{|\Omega|^\frac{N-2}{N-1}}\right)^{N-1}.
\end{equation}
\end{prop}
\begin{proof}
We first observe that it is sufficient to prove \eqref{ue}, then \eqref{ueil} follows by estimating the inradius from below with the aid of Lemma \ref{lm:makai}.
\par
In order to prove \eqref{ue}, we set for simplicity
\[
d=\mathrm{diam}(\Omega).
\]
We then observe that if $d\le 4\,R_\Omega$, then we get
\[
d\le 4\,\frac{N\,\omega_N\,R_\Omega^{N-1}}{N\,\omega_N\,R_\Omega^{N-2}}\le \frac{4}{N\,\omega_N}\,\frac{P(\Omega)}{R_\Omega^{N-2}}.
\]
In the last inequality we used that $N\,\omega\,R_\Omega^{N-1}$ is the perimeter of a largest inscribed ball, together with the usual fact that the perimeter is monotone with respect to set inclusion, on convex sets. Thus \eqref{ue} is proved, under the assumption $d\le 4\,R_\Omega$.
\par
We now suppose that $d>4\,R_\Omega$ and take two points $x,y\in\partial\Omega$ such that $|x-y|=d$. By convexity, we have that $\Omega$ contains the cone obtained as the convex envelope of $\{x\}$ and a ball $B_{R_{\Omega}}(x_0)\subset \Omega$. Similarly, it contains the convex envelope of $\{y\}$ and of the same ball $B_{R_{\Omega}}(x_0)$. If we call $\mathcal{C}(x)$ and $\mathcal{C}(y)$ these two cones, we thus have
\[
\mathcal{C}(x)\cup \mathcal{C}(y)\subset \Omega.
\]
By using again the monotonicity of the perimeter for convex sets, we get that
\begin{equation}
\label{fiuuu}
\max\Big\{P(\mathcal{C}(x)),\,P(\mathcal{C}(y))\Big\}\le P(\Omega).
\end{equation}
By construction, we have that at least one between $|x-x_0|$ and $|y-y_0|$ must be greater than or equal to $d/2$. Without loss of generality, we suppose that $|x-x_0|\ge d/2$. 
Then it is not difficult to see that the perimeter of $P(\mathcal{C}(x))$ can be estimated from below as follows
\[
\begin{split}
P(\mathcal{C}(x))> P(\mathcal{C}(x)\setminus\partial B_R(x_0))&=\omega_{N-1}\,R_\Omega^{N-2}\,|x-x_0|\,\left(1-\frac{R_\Omega^2}{|x-x_0|^2}\right)^\frac{N-1}{2}\\
&\ge \frac{\omega_{N-1}}{2}\,R_\Omega^{N-2}\,d\,\left(1-\frac{4\,R_\Omega^2}{d^2}\right)^\frac{N-1}{2}\\
&> \frac{\omega_{N-1}}{2}\,\left(\frac{3}{4}\right)^\frac{N-1}{2}\,R_\Omega^{N-2}\,d.
\end{split}
\]
In the second inequality, we used the hypothesis $d>4\,R_\Omega$.
\begin{figure}[h]
\includegraphics[scale=.3]{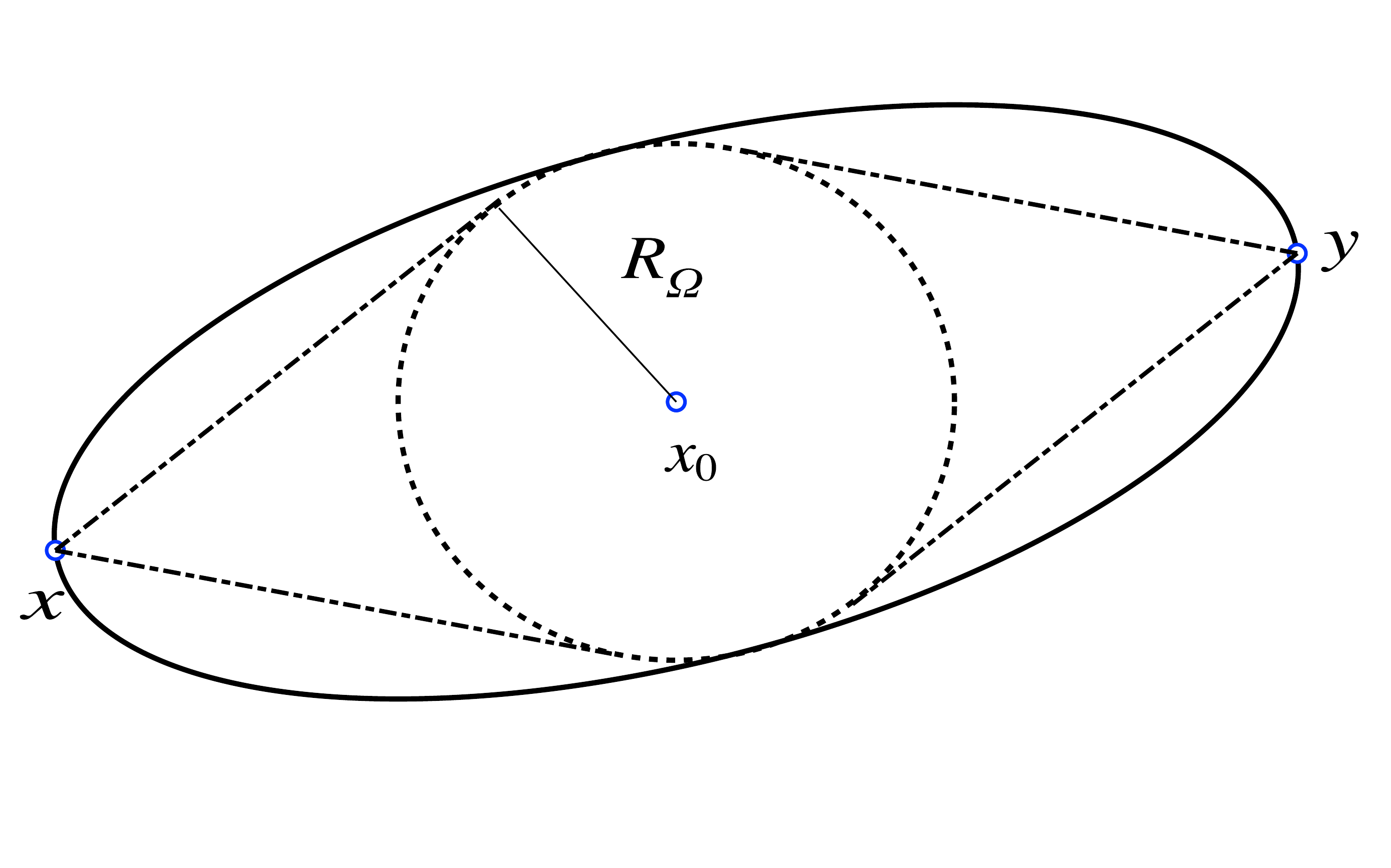}
\caption{The construction for the proof of Proposition \ref{lm:geometrico}, in the case $\mathrm{diam}(\Omega)>4\,R_\Omega$.}
\end{figure}
\par
Thus, by recalling the definition of $d$ and inserting the last estimate in \eqref{fiuuu}, we obtain
\[
\frac{\omega_{N-1}}{2}\,\left(\frac{3}{4}\right)^\frac{N-1}{2}\,R_\Omega^{N-2}\,\mathrm{diam}(\Omega)\le P(\Omega),
\]
which proves \eqref{ue} in the case $\mathrm{diam}(\Omega)>4\,R_\Omega$, as well.
\end{proof}
\begin{oss}[Convexity matters]
We can observe that Proposition \ref{lm:geometrico} does not hold for $N\ge 3$, if the convexity assumption is removed. Indeed, take the sequence of cylinders
\[
T_n=\left\{(x',x_N)\in\mathbb{R}^N\, :\, |x'|<\frac{1}{n}\quad \mbox{ and }\quad |x_n|<n\right\},
\]
for which
\[
\lim_{n\to\infty}\mathrm{diam}(T_n)=+\infty\qquad \mbox{ and }\qquad \lim_{n\to\infty}P(T_n)=\lim_{n\to\infty} \omega_{N-1}\,\frac{1}{n^{N-3}}<+\infty. 
\]
Then define the sets
\[
\Omega_n=B_1(0,\dots,0,-n)\cup T_n\cup B_1(0,\dots,0,n).
\]
We have
\[
|\Omega_n|\ge 2\,\omega_N>0\qquad \mbox{ and }\qquad \lim_{n\to\infty} \mathrm{diam}(\Omega_n)=+\infty,
\]
while
\[
\lim_{n\to\infty} P(\Omega_n)\le 2\,N\,\omega_N+\lim_{n\to\infty}P(T_n)<+\infty.
\]
On the contrary, for $N=2$ the convexity assumption can be dropped. Indeed, in this case the perimeter of a set decreases under convexification, while the diameter is unchanged. Thus the validity of the estimate for bounded convex sets entails that this is valid for general open bounded sets, as well.
\end{oss}
In the proof of Theorem \ref{teo:existence}, we needed the following result for convex sets. If one is not interested in sharp constants, the proof is an easy combination of the estimates above and the isoperimetric inequality. Nevertheless, the final outcome is quite sophisticated, as it mixes four different geometric quantities.
\begin{prop}
\label{prop:inradius}
Let $(N-1)/N<\alpha$ and let $\Omega\subset\mathbb{R}^N$ be an open bounded convex set. Then:
\begin{itemize}
\item if $\alpha<1$ there exists $C_1=C_1(N,\alpha)>0$ such that
\begin{equation}
\label{good}
C_1\,\mathrm{diam}(\Omega)\le R_\Omega^{\alpha\,\frac{N-1}{1-\alpha}}\,\left(\frac{P(\Omega)}{|\Omega|^\alpha}\right)^{\frac{1}{1-\alpha}\,\frac{\alpha\,N-1}{\alpha\,N-(N-1)}};
\end{equation}
\item if $\alpha>1$ there exists $C_2=C_2(N,\alpha)>0$ such that
\begin{equation}
\label{bad}
C_2\,\mathrm{diam}(\Omega)\le \frac{1}{R_\Omega^{\alpha\,\frac{N-1}{\alpha-1}}}\,\left(\frac{|\Omega|^\alpha}{P(\Omega)}\right)^{\frac{1}{\alpha-1}\,\frac{\alpha\,N-1}{\alpha\,N-(N-1)}}
\end{equation}
\end{itemize}
\end{prop}
\begin{proof}
Both estimates are scale invariant, thus since $\alpha>(N-1)/N$ we can assume without loss of generality that
\[
\frac{P(\Omega)}{|\Omega|^\alpha}=1.
\]
Let us assume that $\alpha<1$, by using the second inequality of Lemma \ref{lm:makai}, we then get
\[
1=\frac{|\Omega|^\alpha}{P(\Omega)}\le R_{\Omega}^{\alpha}\,P(\Omega)^{\alpha-1},\qquad \mbox{ i.\,e. }\qquad P(\Omega)\le R_{\Omega}^\frac{\alpha}{1-\alpha}.
\]
On the other hand, by using the isoperimetric inequality, we have
\[
1=\frac{|\Omega|^\alpha}{P(\Omega)}\le \frac{1}{N\,\omega_N^{1/N}}\,|\Omega|^{\alpha-\frac{N-1}{N}},\qquad \mbox{ i.\,e. } \qquad |\Omega|\ge (N\,\omega_N)^\frac{1}{\alpha-\frac{N-1}{N}}.
\]
By spending the last two informations in \eqref{ueil}, we finally get
\[
\left((N\,\omega_N)^\frac{N-2}{\alpha-\frac{N-1}{N}}\,\gamma_N\right)\,\mathrm{diam}(\Omega)\le  R_\Omega^{\alpha\,\frac{N-1}{1-\alpha}},
\]
which gives the desired conclusion.
\vskip.2cm\noindent
The case $\alpha>1$ is similar, we only need to use the first inequality in Lemma \ref{lm:makai}, which now yields
\[
1=\frac{|\Omega|^\alpha}{P(\Omega)}\ge \frac{R_{\Omega}^{\alpha}}{N^\alpha}\,P(\Omega)^{\alpha-1},\qquad \mbox{ i.\,e. }\qquad P(\Omega)\le \frac{N^\frac{\alpha}{\alpha-1}}{R_{\Omega}^\frac{\alpha}{\alpha-1}}.
\]
The rest of the proof goes as before, we leave the details to the reader.
\end{proof}
\begin{oss}[The borderline case $\alpha=1$]
For $\alpha=1$ it is not possible to have a control from above on $\mathrm{diam}(\Omega)$, in terms of
\[
R_\Omega\qquad \mbox{ and }\qquad \frac{P(\Omega)}{|\Omega|}.
\]
Indeed, by still taking the slab--type sequence
\[
\Omega_L=\left(-\frac{L}{2},\frac{L}{2}\right)^{N-1}\times (0,1),
\]
we have seen in \eqref{rapportino} that
\[ 
\frac{P(\Omega_L)}{|\Omega_L|}\sim 2,\qquad \mbox{ for }L\to +\infty,
\]
while $R_{\Omega_L}=1/2$, for $L>1$. On the other hand, it is easily seen that
\[
\mathrm{diam}(\Omega_L)\to +\infty,\qquad \mbox{ for }L\to +\infty.
\]
\end{oss}


\begin{thebibliography}{100}

\bibitem{AC}  F. Alter, V. Caselles, Uniqueness of the Cheeger set of a convex body, Nonlinear Anal., {\bf 70} (2009), 32--44.

\bibitem{AK} D. H. Armitage, \"U. Kuran, The Convexity of a Domain and the Superharmonicity of the Signed Distance Function, Proc. Amer. Math. Soc., {\bf 93} (1985), 598--600.

\bibitem{BFR} L. Brasco, G. Franzina, B. Ruffini, Schr\"odinger operators with negative potentials and Lane-Emden densities, J. Funct. Anal., {\bf 274} (2018), 1825--1863. 

\bibitem{BNT} L. Brasco, C. Nitsch, C. Trombetti, An inequality \`a la Szeg\H{o}-Weinberger for the $p-$Laplacian, Commun. Contemp. Math., {\bf 18} (2016), 1550086, 23 pp.

\bibitem{BR} L. Brasco, B. Ruffini, Compact Sobolev embeddings and torsion functions,  Ann. Inst. H. Poincar\'e Anal. Non Lin\'eaire, {\bf 34} (2017), 817--843.

\bibitem{bu}
P. Buser, A note on the isoperimetric constant,  Ann. Sci. \'Ecole Norm. Sup. (4), {\bf 15} (1982), 213--230. 

\bibitem{cheeger}
J. Cheeger, A lower bound for the smallest eigenvalue of the Laplacian, In \emph{Problems in Analysis} (A symposium in honor of S. Bochner), Princeton University Press (1970), 195--199.

\bibitem{DPG} F. Della Pietra, N. Gavitone, Sharp bounds for the first eigenvalue and the torsional rigidity related to some anisotropic operators, Math. Nachr., {\bf 287} (2014), 194--209. 

%\bibitem{FGL} I. Fragal\`a, F. Gazzola, J. Lamboley, Sharp bounds for the $p-$torsion of convex planar domains, Geometric properties for parabolic and elliptic PDE's, 97--115, Springer INdAM Ser., 2, Springer, Milan, 2013

\bibitem{FL} G. Franzina, P. D. Lamberti, Existence and uniqueness for $p-$Laplacian nonlinear eigenvalue problems, Electron. J. Differential Equations, {\bf 26} (2010), 10 pp.

\bibitem{HP} A. Henrot, M. Pierre, Variation et optimisation de formes. (French) [Shape variation and optimization] 
Une analyse g\'om\'etrique. [A geometric analysis] Math\'ematiques \& Applications (Berlin) [Mathematics \& Applications], {\bf 48}. Springer, Berlin, 2005.

%\bibitem{HL} R. Hynd, E. Lindgren, Extremal functions for Morrey's inequality in convex domains, preprint (2016), available at {\tt https://arxiv.org/abs/1609.08186}

\bibitem{JS} I. Jo\'o, L. Stach\'o, Generalization of an inequality of G. P\'olya concerning the eigenfrequences of vibrating bodies,
Publ. Inst. Math. (Beograd), {\bf 31} (1982), 65--72. 

\bibitem{Le2} M. Ledoux, Spectral gap, logarithmic Sobolev constant, and geometric bounds,
Surv. Differ. Geom., {\bf 9}, Int. Press, Somerville, MA, 2004.

\bibitem{Le}
M. Ledoux, A simple analytic proof of an inequality by P. Buser, Proc. Amer. Math. Soc., {\bf 121} (1994), 951--959. 

\bibitem{LW} L. Lefton, D. Wei, Numerical approximation of the first eigenpair of the p-Laplacian using finite elements and the penalty method, Numer. Funct. Anal. Optim., {\bf 18} (1997), 389--399.  

\bibitem{maz} V. Maz'ya, Sobolev spaces, Sobolev spaces with applications to elliptic partial differential equations. Second, revised and augmented edition. Grundlehren der Mathematischen Wissenschaften [Fundamental Principles of Mathematical Sciences], {\bf 342}. Springer, Heidelberg, 2011. 

\bibitem{Pa} E. Parini, Reverse Cheeger inequality for planar convex sets, J. Convex Anal., {\bf 24} (2017), 107--122.

\bibitem{Po} G. P\'{o}lya, Two more inequalities between physical and geometrical quantities, J. Indian Math. Soc., {\bf 24} (1960), 413--419.

\bibitem{RW} X. Ren, J. Wei, Counting peaks of solutions to some quasilinear elliptic equations with large exponents, J. Differential Equations, {\bf 117} (1995), 28--55.

\bibitem{Sch} R. Schneider, Convex bodies: the Brunn-Minkowski theory. Second expanded edition. Encyclopedia of Mathematics and its Applications, {\bf 151}. Cambridge University Press, Cambridge, 2014.

\end{thebibliography}
\end{document}